\documentclass[12pt]{amsart}

\usepackage{amsmath,amssymb,amsthm} 

\newtheorem*{mainthm}{Theorem}

\newtheorem{cor}{Corollary}

\newcommand{\et}{\quad\mbox{and}\quad}

\newcommand{\bQ}{\mathbb{Q}}
\newcommand{\bR}{\mathbb{R}}
\newcommand{\bZ}{\mathbb{Z}}
\newcommand{\disp}{\displaystyle}

\newcommand{\omegal}{\omega^{\mathrm{lead}}}
\newcommand{\Qbar}{\overline{\bQ}}
\newcommand{\Sl}{S^{\mathrm{lead}}}
\newcommand{\ux}{\mathbf{x}}

\begin{document}

\title{A transference inequality for rational approximation to points in geometric progression}

\author{J\'er\'emy Champagne}
\address{
   D\'epartement de Math\'ematiques\\
   Universit\'e d'Ottawa\\
   150 Louis Pasteur\\
   Ottawa, Ontario K1N 6N5, Canada}
\email{jcham016@uottawa.ca}
\author{Damien Roy}
\address{
   D\'epartement de Math\'ematiques\\
   Universit\'e d'Ottawa\\
   150 Louis Pasteur\\
   Ottawa, Ontario K1N 6N5, Canada}
\email{droy@uottawa.ca}
\thanks{Work of the authors partially supported by NSERC}


\begin{abstract}
\noindent We establish a transference inequality conjectured by Badziahin and Bugeaud
relating exponents of rational approximation of points in geometric progression.
\end{abstract}

\maketitle

Let $\xi\in\bR$.  For each integer $n\ge 1$, one defines
$\omega_n(\xi)$ as the supremum of all $\omega\in\bR$  for which there exist
infinitely many non-zero polynomials $P(x)\in\bZ[x]$ of degree at most $n$ with
\[
 |P(\xi)|\le \|P\|^{-\omega},
\]
where $\|P\|$ stands for the largest absolute value
of the coefficients of $P$.  These quantities form the basis of Mahler's
classification of numbers \cite{Ma1932}.  Dually, one defines $\lambda_n(\xi)$ as
the supremum of all $\lambda\in\bR$ for which there exist infinitely many
non-zero points $\ux=(x_0,\dots,x_n)\in\bZ^{n+1}$ such that
\[
 \max_{1\le m\le n} |x_0\xi^m-x_m| \le \|\ux\|^{-\lambda},
\]
where $\|\ux\|$ stands for the maximum norm of $\ux$.

Suppose now that $\xi$ is transcendental over $\bQ$.  In \cite{BB2019},
D.~Badziahin and Y.~Bugeaud prove that for any integers $k$, $n$
with $2\le k\le n$, we have
\begin{equation}
 \label{eq1}
 \lambda_n(\xi)\ge \frac{\omegal_k(\xi)-n+k}{(k-1)\omegal_k(\xi)+n}
\end{equation}
where $\omegal_k(\xi)$ is defined as $\omega_k(\xi)$ but by restricting to polynomials
$P(x)\in\bZ[x]$ of degree at most $k$ whose coefficient $c_k(P)$ of $x^k$ has largest
absolute value $|c_k(P)|=\|P\|$.  They conjecture that this inequality remains true if
$\omegal_k(\xi)$ is replaced by $\omega_k(\xi)$ and they prove this is indeed the case
if $k=2$ or $k=n-1$.  Their proof for $k=2$ is based on the fact that $\lambda_n(\xi)=
\lambda_n(1/\xi)$.  The purpose of this note is to prove this conjecture based on their
inequality \eqref{eq1}, the invariance of $\lambda_n$ and $\omega_k$ by general
fractional linear transformations with rational coefficients, and the following observation.

\begin{mainthm}
Let $k\ge 1$ be an integer and let $r_0,\dots,r_k$ be distinct integers.
There is an integer $M\ge 1$ such that, for each
$\xi\in\bR$, there exists at least one index $i\in\{0,1,\dots,k\}$,
with $r_i\neq\xi$, for which the point $\xi_i=1/(M(\xi-r_i))$
satisfies $\omegal_k(\xi_i)=\omega_k(\xi_i)=\omega_k(\xi)$.
\end{mainthm}

\begin{proof}
We first note that there exists positive constants $C_1$
and $C_2$ such that
\[
 \|P\|\le C_1\max\{|P(r_i)|\,;\,0\le i\le k\}
 \et
 \max\{\|P(x+r_i)\|\,;\,0\le i\le k\} \le C_2 \|P\|
\]
for any polynomial $P\in\bR[x]$ of degree at most $k$.
Let $\xi\in\bR$. Choose an integer $M$ with $M\ge C_1C_2$ and
a sequence of polynomials $(P_j)_{j\ge 1}$ in $\bZ[x]$
of degree at most $k$ with strictly increasing norms such that
\begin{equation}
 \label{eq2}
 \lim_{j\to\infty} -\frac{ \log|P_j(\xi)|}{\log\|P_j\|} = \omega_k(\xi).
\end{equation}
Then, choose $i\in\{0,1,\dots,k\}$ such that $\|P_j\|\le C_1 |P_j(r_i)|$ for an
infinite set $S$ of positive integers $j$, and set
\[
 Q_j(x)=(Mx)^kP_j\Big(\frac{1}{Mx}+r_i\Big) \in \bZ[x]
\]
for each $j\in S$.  For those values of $j$, the absolute value of the coefficient
of $x^k$ in $Q_j(x)$ is $|c_k(Q_j)|=M^k|P_j(r_i)|$ while its other coefficients
have absolute value at most
\[
 M^{k-1}\|P_j(x+r_i)\|
  \le C_2M^{k-1}\|P_j\|
  \le C_1C_2M^{k-1} |P_j(r_i)|
  \le M^{k} |P_j(r_i)|,
\]
thus $|c_k(Q_j)|=\|Q_j\|$.  We also have $r_i\neq\xi$ and
\[
 |Q_j(\xi_i)|=(M|\xi_i|)^k|P_j(\xi)|
\]
where $\xi_i=1/(M(\xi-r_i))$.  As the ratio $\|Q_j\|/\|P_j\|$ is
bounded from above and from below by positive constants, we
deduce from \eqref{eq2} that $-\log|Q_j(\xi_i)|/\log\|Q_j\|$
converges to $\omega_k(\xi)$ as $j$ goes to infinity in $S$.
Altogether, this means that $\omegal_k(\xi_i)\ge \omega_k(\xi)$.
However, it is well known (and easy to prove) that
$\omega_k(\xi)=\omega_k(\xi_i)$ because $\xi_i$ is the image
of $\xi$ by a linear fractional transformation with rational
coefficients.  The conclusion follows because
$\omega_k(\xi_i)\ge\omegal_k(\xi_i)$ by the very definition
of $\omegal_k$.
\end{proof}

Applying the formula \eqref{eq1} with $\xi$ replaced by $\xi_i$ and using the fact
that $\lambda_n(\xi_i)=\lambda_n(\xi)$, we reach the desired inequality.

\begin{cor}
$\disp \lambda_n(\xi)\ge \frac{\omega_k(\xi)-n+k}{(k-1)\omega_k(\xi)+n}$\quad
$(\xi\notin\Qbar \ \text{and}\ 2\le k\le n)$.
\end{cor}

We thank Yann Bugeaud and Victor Beresnevich for pointing out that our theorem 
formalizes principles that are implicit in Lemmas 3 and 4 of \cite{Ba1976} as well 
as on pages 25--26 of \cite{Sp1969} (for the choice of $r_i=i$). The next corollary 
is an application to metrical theory that was suggested by Yann Bugeaud.

\begin{cor}
Let $k\ge 1$ be an integer and let $w\in\bR$.  The sets
$S=\{\xi\in\bR\,;\,\omega_k(\xi)\ge w\}$ and
$\Sl=\{\xi\in\bR\,;\,\omegal_k(\xi)\ge w\}$ have the same Hausdorff
dimension.  This remains true if we replace the large inequalities
$\ge $ by strict inequalities $>$ in the definitions of $S$ and $\Sl$. 
\end{cor}

\begin{proof}
We have $\Sl\subset S$ and, for an appropriate choice of integers
$r_0,\dots,r_k$ and $M$, the theorem gives
$S\subseteq \bigcup_{i=0}^k \tau_i^{-1}(\Sl)$ where
$\tau_i(\xi)=1/(M(\xi-r_i))$ for each
$\xi\in\bR\setminus\{r_i\}$ $(0\le i\le k)$.
The conclusion follows by the invariance of the Hausdorff
dimension under fractional linear transformations.
\end{proof}

\end{document}